\documentclass[11pt]{amsart}

\usepackage[letterpaper,margin=1in]{geometry}
\usepackage{amsmath,amssymb,amsthm,mathtools}
\usepackage{microtype}
\usepackage[colorlinks=true,linkcolor=black,citecolor=black,urlcolor=black]{hyperref}

\numberwithin{equation}{section}

\theoremstyle{plain}
\newtheorem{theorem}{Theorem}
\newtheorem{lemma}[theorem]{Lemma}
\newtheorem{proposition}[theorem]{Proposition}
\newtheorem{corollary}[theorem]{Corollary}

\theoremstyle{remark}
\newtheorem{remark}[theorem]{Remark}

\newcommand{\ZZ}{\mathbb Z}
\newcommand{\RR}{\mathbb R}
\newcommand{\CC}{\mathbb C}
\newcommand{\Zn}[1]{\mathbb Z_{#1}}

\newcommand{\ch}[1]{\chi_{#1}}
\newcommand{\wh}[1]{\widehat{#1}}
\newcommand{\psin}{\psi_n}

\newcommand{\Apsi}{A_{\psin}}

\newcommand{\Pt}{P_t}

\newcommand{\Lcyc}{L_n}
\newcommand{\coef}{c}

\newcommand{\Ssymbol}{S}
\newcommand{\Qsymbol}{Q}
\newcommand{\Ssq}{\Ssymbol^2}
\newcommand{\Qsq}{\Qsymbol^2}
\newcommand{\Sfour}{\Ssymbol^4}

\newcommand{\TT}{\mathbb T}
\newcommand{\muT}{\pi}
\newcommand{\AT}{A_{\TT}}
\newcommand{\AvgT}[1]{\int_{\TT} #1\,d\muT}
\newcommand{\EntT}[1]{\operatorname{Ent}_{\muT}\!\left(#1\right)}

\newcommand{\AvgSymbol}{\pi}
\newcommand{\Avg}[1]{\AvgSymbol\!\left(#1\right)}

\newcommand{\EntMeasure}{\pi}
\newcommand{\Ent}[1]{\operatorname{Ent}_{\EntMeasure}\!\left(#1\right)}
\newcommand{\Energy}{\mathcal E}
\newcommand{\En}{\Energy_n}
\newcommand{\norm}[1]{\left\|#1\right\|}
\newcommand{\abs}[1]{\left|#1\right|}

\newcommand{\ii}{\mathrm i}

\title{Sharp log-Sobolev inequalities and quartic stability on finite cyclic groups}

\author[X. Xie]{Xinyuan Xie}
\address{(X.X.) Department of Mathematics, University of California, Irvine, CA 92697, USA}
\email{xinyuax7@uci.edu}

\author[H. Zhang]{Haonan Zhang}
\address{(H.Z.) Department of Mathematics, University of South Carolina, Columbia, SC 29208, USA}
\email{haonanzhangmath@gmail.com}

\date{\today}

\hypersetup{
  pdftitle={Sharp log-Sobolev inequalities and quartic stability on finite cyclic groups},
  pdfauthor={Xinyuan Xie and Haonan Zhang},
  pdfsubject={Sharp log-Sobolev inequalities for word-length Laplacians},
  pdfkeywords={log-Sobolev inequalities, hypercontractivity, finite cyclic groups, Poisson semigroups, stability of log-Sobolev inequality}
}

\begin{document}

\begin{abstract}
Let $\mathbb Z_n$ be the cyclic group equipped with the uniform probability
measure $\pi$, and let $A_{\psi_n}$ be the Laplacian with word length
$$
  \psi_n(k) = \min(k,n-k).
$$
For every $n\ge4$, we prove the sharp log-Sobolev inequality
$$
  \text{Ent}_{\pi}(|f|^2)
  \le 2\pi(\bar{f}A_{\psi_n} f),
  \qquad f:\mathbb Z_n \to \mathbb{C},
$$
where $\text{Ent}_{\pi}$ is the relative entropy with respect to $\pi$.  Equivalently, the Poisson semigroup $P_t=e^{-tA_{\psi_n}}$ satisfies the optimal hypercontractivity.
The proof is inspired by the recent work of Frank and Ivanisvili [FI26] on a sharp
log-Sobolev inequality for the nearest-neighbor simple random walk.
Similar arguments yield a simple proof of Weissler's sharp log-Sobolev inequality for the Poisson semigroup on
the circle \cite{Weissler1980}.
The same inequalities were independently obtained by
Yao~\cite{Yao2026} using a different method.

We also prove quantitative stability estimates.
For $n\ge 4$ and $f:\mathbb Z_n\to[0,\infty)$ with
$\pi(f^2)=1$, 
$$
  2\pi(fA_{\psi_n}f)-\operatorname{Ent}_{\pi}(f^2)
  \ge \frac1{12}\|f-1\|_{L^2(\pi)}^4,
$$
with coefficient $1/(12d)$ on products $(\mathbb Z_n)^d$.
The quartic
order and the $d^{-1}$ dependence are optimal.

\end{abstract}


\maketitle

\section{Introduction}

There are two natural translation-invariant dynamics on the discrete circle, or cyclic group,
\(\Zn{n}\).  The first is the nearest-neighbor simple random walk: from a
point, one jumps to either neighbor with equal probability.  The associated
positive Laplacian is
\[
  \Lcyc (f)(k)=f(k)-\frac{1}{2}(f(k+1)+f(k-1)).
\]
With
\[
  \wh f(k):=\Avg{f\,\overline{\ch{k}}}
  =\frac1n\sum_{x\in\Zn{n}}f(x)e^{-2\pi\ii kx/n},
\]
the Fourier expansion is
\[
f=\sum_{k\in \Zn{n}} \wh f(k)\ch{k}, \qquad  \ch{k}(x):=e^{2\pi \ii kx/n},
\]
and $L_n$ is the Fourier multiplier
\[
L_n(f)=\sum_{k\in \Zn{n}}r_n(k)\wh f(k)\ch{k},\qquad r_n(k):=1-\cos(2\pi k/n).
\]

The second can be described by the positive word-length Laplacian $\Apsi$, which is the Fourier
multiplier
\[
  \Apsi f=\sum_{k\in \Zn{n}}\psin(k)\wh f(k)\ch{k},
  \qquad
  \psin(k):=\min(k,n-k).
\]
Here and below, the last formula uses the representative
\(k\in\{0,\ldots,n-1\}\). Put \(m=\lfloor n/2\rfloor\). The corresponding
positive symmetric jump-kernel representation is
\begin{equation}\label{eq:jump-kernel}
  \Apsi f(x)
  =\sum_{r\in\Zn{n}\setminus\{0\}}q_n(r)\bigl(f(x)-f(x+r)\bigr),
  \qquad
  q_n(r):=
  \frac{1-\cos(2\pi mr/n)}{n\bigl(1-\cos(2\pi r/n)\bigr)}\ge0.
\end{equation}
The conditional negativity of this cyclic word length is also established in
\cite[Appendix~B]{JPPP2017}. Consequently, $-\Apsi$ is the generator of a
symmetric Markov semigroup, which we call the Poisson semigroup.
Here, one can think of $\psin$ as the word length function with respect to the symmetric generator set $\{-1,1\}$.

In this paper, we shall consider sharp log-Sobolev inequalities, or equivalently hypercontractivity, in these two cases. For classical log-Sobolev inequalities and hypercontractivity, we refer to
\cite{Bonami1970,Gross1975,Beckner1975,Gross2006} and references therein for background.

We shall use the convention that \(\pi\) denotes the uniform probability measure
on \(\Zn{n}\).  Thus, for all $f$ on $\Zn{n}$,
\[
\Avg{f}
=
\frac1n\sum_{x\in\mathbb Z_n}f(x).
\]
For all nonnegative $f$ on $\Zn{n}$, we define
\[
\Ent{f}
:=
\Avg{f\log f}
-
\Avg{f}
\log \Avg{f}.
\]
We use the convention \(0\log0=0\). For \(1\le p<\infty\), we write
\(\norm{f}_p=\Avg{|f|^p}^{1/p}\). For \(f:\Zn{n}\to\CC\), we define the
Dirichlet form
\[
  \En(f):=\Avg{\overline f\,\Apsi f}
  =\sum_{k=0}^{n-1}\psin(k)\abs{\wh f(k)}^2.
\]
Because the kernel in \eqref{eq:jump-kernel} is nonnegative, the reverse
triangle inequality gives \(\En(|f|)\le\En(f)\). Thus every log-Sobolev
inequality below for complex-valued functions follows from its nonnegative
real-valued form. We therefore prove and state the inequalities in that form
without repeating this reduction. The same modulus-contraction principle
applies to the other Markov forms considered below, including the Poisson form
on the circle.

For the nearest-neighbor cycle, it is enough to consider the log-Sobolev
inequality for nonnegative $f:\Zn{n}\to\RR$:
\begin{equation}\label{lsi1}
    \Ent{f^2}\le c\Avg{f L_n f}.
\end{equation}
The problem of determining the sharp
log-Sobolev constant $c$ goes back to the work of Diaconis and Saloff-Coste on
finite Markov chains \cite{DiaconisSaloffCoste1996}.
After partial progress by Chen and Sheu \cite{ChenSheu2003}, Chen, Liu and
Saloff-Coste \cite{ChenLiuSaloffCoste2008}, and Faust and Fawzi
\cite{FaustFawzi2024}, Frank and Ivanisvili settled the problem for all
\(n\ge4\), proving
that the optimal constant is \(c=2/r_n(1)\)
\cite{FrankIvanisvili2026}. Their key new step is an optimal cubic Sobolev
inequality, which inspired this work.

For the word-length Laplacian, the log-Sobolev inequality likewise reads, for
nonnegative $f:\Zn{n}\to\RR$,
\begin{equation}\label{lsi2}
    \Ent{f^2}\le c\Avg{f \Apsi f}.
\end{equation}
The associated hypercontractivity problem, including several sharp special
cases and exponent ranges, was studied by Beckner, Janson and
Jerison~\cite{BJJ1983}, Andersson~\cite{Andersson2002}, Wolff~\cite{Wolff2007},
and Junge, Palazuelos, Parcet and Perrin~\cite{JPPP2017}.
The optimal hypercontractivity for $n=3$ is special, and the optimal constants for \(\Zn{3}\), for all
\(1<p<q<\infty\), were determined by Cao, Fan, Han, Qiu and
Wang~\cite{CaoFanHanQiuWang2026}. For \(n\ge4\), Yao first settled the
cases \(n=2^r\) and
\(n=3\cdot2^r\)~\cite{Yao2025}, and subsequently the full range
\(n\ge4\)~\cite{Yao2026}.

After normalizing both generators to have spectral gap one, the word-length
inequality~\eqref{lsi2} is stronger than the nearest-neighbor
inequality~\eqref{lsi1}. Indeed, both generators share the same eigenvectors,
the characters $\chi_k$, and their eigenvalues satisfy
\begin{equation}\label{eq:eigenvalue-comparison}
    \ell_n(k):=\frac{r_n(k)}{r_n(1)}\ge \psi_n(k),\qquad 1\le k\le n-1.
\end{equation}
To see this, write \(m=\min\{k,n-k\}\le n/2\), set
\(\theta=2\pi/n\), and let \(F(t)=1-\cos t\).  The identity
\[
  F(a+b)-F(a)-F(b)
  =4\cos\!\left(\frac{a+b}{2}\right)
   \sin\!\left(\frac a2\right)
   \sin\!\left(\frac b2\right)
  \ge0,
  \qquad a,b\ge0,\quad a+b\le\pi,
\]
shows that \(F\) is superadditive on such pairs.  We now prove
\(F(m\theta)\ge mF(\theta)\) by induction on \(m\).  The case \(m=1\)
is immediate.  If \(m\ge2\), then \(m\theta\le\pi\), and the induction
hypothesis gives
\[
  F(m\theta)
  \ge F((m-1)\theta)+F(\theta)
  \ge mF(\theta).
\]
Consequently,
\[
  \frac{r_n(k)}{r_n(1)}
  =\frac{F(m\theta)}{F(\theta)}
  \ge m=\psi_n(k),
\]
where we used \(F(k\theta)=F(m\theta)\).  This proves
\eqref{eq:eigenvalue-comparison}.

\subsection{Log-Sobolev inequalities for the word-length Laplacian on
\texorpdfstring{\(\Zn{n}\)}{Z\_n}}

One of our main results is a new proof of the sharp log-Sobolev inequality for the
word-length Laplacian on \(\mathbb Z_n\) for \(n\ge4\).
Yao obtained the result before us~\cite{Yao2026}. Our proof was carried out
independently of his work and uses a substantially different method.
Our approach is inspired by the work of Frank and Ivanisvili \cite{FrankIvanisvili2026}.

\begin{theorem}\label{thm:main}
For every \(n\ge4\) and every nonnegative \(f:\Zn{n}\to\RR\),
\[
  \Ent{f^2}\le 2\En(f).
\]
The constant \(2\) is sharp.
\end{theorem}

Following standard arguments, one obtains the following hypercontractivity
result for the Poisson semigroup \(\Pt=e^{-t\Apsi}\). Sufficiency follows from Gross's theorem
\cite{Gross1975} applied to Theorem \ref{thm:main}; necessity follows from the
standard perturbation.

\begin{corollary}[Hypercontractivity]\label{cor:hypercontractivity}
Fix \(n\ge4\), and let \(\Pt=e^{-t\Apsi}\). For \(1<p\le q<\infty\) and
\(t\ge0\),
\[
  \norm{\Pt f}_{q}\le \norm{f}_{p}\qquad \text{for all }f:\Zn{n}\to\CC
\]
holds if and only if
\[
  t\ge \frac12\log\frac{q-1}{p-1}.
\]
\end{corollary}

\begin{remark}
    By a standard tensorization argument, the above log-Sobolev inequality and
    hypercontractivity extend to products of $\mathbb{Z}_n$ with the same
    constants; see, for instance, \cite{Gross2006}.
\end{remark}

Our proof of Theorem \ref{thm:main} uses the cubic-majorant reduction idea of Frank and Ivanisvili, who observed that
\begin{equation}\label{eq:cubic-intro}
    2t^2\log t
  \le
  \frac23(t-1)^2(t+2)+(t^2-1),\qquad t\ge 0.
\end{equation}
See Lemma \ref{lem:cubic-majorant} below. This reduces the problem to a cubic, or third-moment, Sobolev inequality.
However, the final high-frequency estimate in their nearest-neighbor proof
does not adapt directly to the word-length multiplier. In fact, the main
obstacle to a literal adaptation of their argument appears in the proof of Lemma~3 in~\cite{FrankIvanisvili2026}, which uses
\[
\sum_{j=2}^{n-2}\frac{1}{\ell_n(j)-1}<\frac{3}{2}.
\]

For the word-length multiplier, however, $\psin(j)-1=j-1$ when $j<n/2$, so the analogous estimate would contain
\[
  \sum_{j=2}^{m}\frac1{j-1}\simeq\log m,
\]
which produces a logarithmic loss.

Our replacement for this final step is a blockwise third-moment estimate. On the Fourier side, we organize
the proof by conjugate pairs of frequencies \(\{j,-j\}\), except in the even case
where the middle frequency \(n/2\) is self-conjugate.  If \(u:\Zn{n}\to \RR\) has mean zero and
\begin{equation}\label{SQ}
     \Ssq=\Avg{u^2},\qquad \Qsq=\Avg{u\Apsi u}-\Ssq,
     \qquad \Ssymbol,\Qsymbol\ge0,
\end{equation}
then we prove
\[
  \abs{\Avg{u^3}}\le 3\Ssq\Qsymbol.
\]
This third-moment estimate counts only the cubic Fourier resonances that survive averaging.
For odd cyclic groups these resonances are
\[
  i+j=k,
  \qquad
  i+j+k=n,
  \qquad
  1\le i,j,k\le\left\lfloor\frac n2\right\rfloor.
\]
For even cyclic groups the same resonances occur, and one must also handle the
self-conjugate middle frequency \(n/2\). Compared with the proof of Frank and Ivanisvili, the main addition is this blockwise third-moment estimate: it closes the cubic inequality for the linear word-length spectrum without summing the harmonic reciprocal-surplus weights.

\subsection{Log-Sobolev inequality for the Poisson semigroup on the circle}

Similar ideas yield a simple proof of the sharp log-Sobolev inequality for the
Poisson semigroup on the circle. Let \(\TT=\RR/\ZZ\), let \(\muT\) be
normalized Haar measure, and let \(\AT\) be the positive Poisson Laplacian on
\(\TT\), so that \(-\AT\) generates the Poisson semigroup:
\[
  \AT (f)=\sum_{k\in\ZZ}\abs{k}\,\wh f(k)\chi_k,
  \qquad
  \chi_k(x)=e^{2\pi\ii kx},
  \qquad
  \wh f(k)=\AvgT{f\,\overline{\chi_k}},
\]
so that its closed quadratic form is
\[
  \mathcal E_{\TT}(f):=\AvgT{\overline f\,\AT f}
  =\sum_{k\in\ZZ}\abs{k}\,\abs{\wh f(k)}^2,
\]
with form domain
\[
  \mathcal D(\mathcal E_{\TT})
  =\left\{f\in L^2(\TT):
    \sum_{k\in\ZZ}\abs{k}\,\abs{\wh f(k)}^2<\infty\right\}.
\]
For nonnegative \(g\), entropy with respect to \(\muT\) is defined by
\[
  \operatorname{Ent}_{\muT}(g)
  =\AvgT{g\log g}-\AvgT{g}\log\AvgT{g},
\]
again with \(0\log0=0\).

Weissler~\cite{Weissler1980} first proved the sharp log-Sobolev inequality for the Poisson semigroup on the circle.

\begin{theorem}\label{thm:circle}
For every nonnegative \(f\in\mathcal D(\mathcal E_{\TT})\),
\[
  \EntT{f^2}\le 2\mathcal E_{\TT}(f).
\]
The constant \(2\) is sharp.
\end{theorem}

This log-Sobolev inequality for the Poisson semigroup can be recovered from
the finite word-length inequalities by a Riemann-sum limit; see
\cite{Yao2026}. Frank and Ivanisvili used the analogous limiting argument for
the nearest-neighbor multiplier~\cite{FrankIvanisvili2026}. Indeed, after
spectral-gap normalization and under the natural identification of fixed
frequencies, the nearest-neighbor multiplier
\[
  \ell_n(k)=\frac{1-\cos(2\pi k/n)}{1-\cos(2\pi/n)}
\]
satisfies \(\ell_n(k)\to k^2\) for each fixed \(k\in\ZZ\), so its
limiting Dirichlet form is \(\sum_{k\in\ZZ}k^2\abs{\wh f(k)}^2\), corresponding
to the heat semigroup. By contrast, the word-length multiplier satisfies
\(\psin(k)\to\abs{k}\) for each fixed \(k\in\ZZ\). Thus the word-length limit
yields the Poisson Dirichlet form
\(\sum_{k\in\ZZ}\abs{k}\,\abs{\wh f(k)}^2\). The corresponding sharp
log-Sobolev inequality for the heat semigroup on \(\TT\) goes back to
Weissler~\cite{Weissler1980} and Rothaus~\cite{Rothaus}.

As in the finite cyclic-group case, the direct third-moment argument extends
to the circle group. We provide the details in Section~\ref{sec:proofs}.

\subsection{Stability of log-Sobolev inequalities}

Our last result is a stability result for the log-Sobolev inequality.
We first state the result for finite cyclic groups. A simple remainder analysis of the cubic-majorant
\eqref{eq:cubic-intro} yields a quartic stability estimate in dimension one, and a tensorization argument gives the higher-dimensional form.

For \(n\ge4\) and \(d\ge1\), write
\[
  G=(\Zn{n})^d,\qquad
  \mu=\pi^{\otimes d},\qquad
  A=A_n^{(d)}:=\sum_{i=1}^d\Apsi^{(i)}.
\]
Here \(\Apsi^{(i)}\) acts as \(\Apsi\) in the \(i\)-th coordinate; that is,
\[
  \bigl(\Apsi^{(i)}h\bigr)(x_1,\ldots,x_d)
  =\left[
    \Apsi h(x_1,\ldots,x_{i-1},\mathord\cdot,x_{i+1},\ldots,x_d)
   \right](x_i).
\]

\begin{theorem}[Quartic stability]\label{thm:stability-quartic}
For every \(n\ge4\) and every \(f:\Zn{n}\to[0,\infty)\) satisfying
\(\Avg{f^2}=1\),
\begin{equation}\label{eq:finite-quartic-stability}
  2\En(f)-\Ent{f^2}
  \ge\frac1{12}\norm{f-1}_2^4.
\end{equation}
More generally, for every \(d\ge1\), every \(f:G\to[0,\infty)\) satisfying
\(\int_G f^2\,d\mu=1\) obeys
\begin{equation}\label{eq:product-quartic-stability}
  2\int_G fAf\,d\mu-\operatorname{Ent}_{\mu}(f^2)
  \ge \frac1{12d}\norm{f-1}_{L^2(\mu)}^4.
\end{equation}
Moreover, the exponent \(4\) in
\eqref{eq:finite-quartic-stability} and the \(d^{-1}\) dependence in
\eqref{eq:product-quartic-stability} are optimal.
\end{theorem}

\begin{remark}[Unnormalized one-dimensional form]
\label{rem:unnormalized-stability}
Applying \eqref{eq:finite-quartic-stability} to \(f/\norm{f}_2\), we obtain,
for every nonzero \(f:\Zn{n}\to[0,\infty)\),
\begin{equation}\label{eq:finite-quartic-stability-unnormalized}
  2\En(f)-\Ent{f^2}
  \ge
  \frac{\norm{f-\norm{f}_2\mathbf1}_2^4}{12\norm{f}_2^2}.
\end{equation}
\end{remark}

The proof of Theorem~\ref{thm:stability-quartic} is included in Section~\ref{sec:stability}.

\bigskip 

A Lean formalization of Theorem~\ref{thm:main} is available in~\cite{XZlean}, and the authors acknowledge the use
of Codex.

\subsection*{Acknowledgments}
H.Z. is supported by NSF DMS-2453408. The authors acknowledge the use of
OpenAI Codex and ChatGPT 5.5 Pro.

\section{The cubic reduction}

We recall the cubic majorant used in
Frank and Ivanisvili's cubic Sobolev reduction \cite{FrankIvanisvili2026}.

\begin{lemma}[Cubic majorant {\cite{FrankIvanisvili2026}}]\label{lem:cubic-majorant}
For every \(t\ge0\),
\[
  2t^2\log t
  \le
  \frac23(t-1)^2(t+2)+(t^2-1),
\]
where the left side is interpreted as \(0\) at \(t=0\).
\end{lemma}

\begin{proposition}\label{prop:cubic-criterion}
Suppose that for every \(f:\Zn{n}\to[0,\infty)\) with \(\Avg{f^2}=1\),
\begin{equation}\label{eq:cubic-criterion}
  \En(f)\ge \frac13\Avg{(f-1)^2(f+2)}.
\end{equation}
Then
\[
  \Ent{f^2}\le 2\En(f)
\]
for every nonnegative \(f:\Zn{n}\to\RR\).
\end{proposition}

\begin{proof}
If \(f\equiv0\), the conclusion is immediate. Otherwise, by homogeneity,
normalize \(\Avg{f^2}=1\). Then
\[
  \Ent{f^2}=2\Avg{f^2\log f}.
\]
Applying Lemma \ref{lem:cubic-majorant} pointwise and averaging gives
\[
  \Ent{f^2}
  \le \frac23\Avg{(f-1)^2(f+2)}
  +\Avg{f^2-1}
  =\frac23\Avg{(f-1)^2(f+2)}.
\]
The conclusion follows from \eqref{eq:cubic-criterion}.
\end{proof}

The next proposition isolates the estimate needed for the cubic criterion \eqref{eq:cubic-criterion}.

\begin{proposition}\label{prop:block-to-cubic}
Fix \(n\ge4\).  Assume that every real mean-zero \(u:\Zn{n}\to\RR\)
satisfies
\begin{equation}\label{eq:block-estimate-abstract}
  \abs{\Avg{u^3}}
  \le 3\Ssq\Qsymbol,
  \qquad
  \Ssq=\Avg{u^2},
  \quad
  \Qsq=\En(u)-\Ssq,
  \quad
  \Ssymbol,\Qsymbol\ge0.
\end{equation}
Then \eqref{eq:cubic-criterion} holds on \(\Zn{n}\).
\end{proposition}

\begin{proof}
Let \(f:\Zn{n}\to[0,\infty)\) and normalize \(\Avg{f^2}=1\).  Put
\[
  a=\Avg{f},\qquad u=f-a.
\]
Then \(\Avg{u}=0\), \(0\le a\le1\), and
\[
  \Ssq=\Avg{u^2}=1-a^2,
  \qquad
  \En(f)=\En(u)=\Ssq+\Qsq.
\]
Since \((t-1)^2(t+2)=t^3-3t+2\),
\[
  \Avg{(f-1)^2(f+2)}=\Avg{f^3}-3a+2.
\]
Expanding \(f=a+u\) gives
\[
  \Avg{f^3}=a^3+3a\Ssq+\Avg{u^3}
  =a^3+3a(1-a^2)+\Avg{u^3}.
\]
Hence
\begin{equation}\label{eq:cubic-expansion}
  \Avg{(f-1)^2(f+2)}=2-2a^3+\Avg{u^3}.
\end{equation}
Therefore \eqref{eq:cubic-criterion} is equivalent to
\begin{equation}\label{eq:needed-after-expansion}
  \Avg{u^3}\le (1-a)^2(1+2a)+3\Qsq.
\end{equation}
By \eqref{eq:block-estimate-abstract},
\[
  \Avg{u^3}\le 3\Ssq\Qsymbol.
\]
Also
\[
  3\Ssq\Qsymbol\le \frac34\Sfour+3\Qsq,
\]
and, since \(\Ssq=1-a^2\),
\[
  (1-a)^2(1+2a)-\frac34\Sfour
  =(1-a)^2(1+2a)-\frac34(1-a^2)^2
  =\frac14(1-a)^3(1+3a)\ge0.
\]
Thus \eqref{eq:needed-after-expansion} holds.
\end{proof}

\section{The third-moment estimate for odd cyclic groups}

Assume \(n=2m+1\ge5\).  For a real mean-zero function
\(u:\Zn{n}\to\RR\), write
\[
  u=\sum_{j=1}^{m}\left(\coef_j\ch{j}+\overline{\coef_j}\ch{-j}\right),
\]
where \(\ch{-j}=\ch{n-j}\).
Define
\[
  a_j=\sqrt2\,\abs{\coef_j},\qquad 1\le j\le m.
\]
Then
\[
  \Ssq=\Avg{u^2}=\sum_{j=1}^{m}a_j^2,
  \qquad
  \Qsq=\sum_{j=2}^{m}(j-1)a_j^2,
  \qquad \Ssymbol,\Qsymbol\ge0.
\]

\begin{proposition}[Odd third-moment estimate]\label{thm:odd-block}
For every real mean-zero \(u:\Zn{2m+1}\to\RR\), \(m\ge2\),
\[
  \abs{\Avg{u^3}}
  \le 3\Ssq\Qsymbol.
\]
\end{proposition}

\begin{proof}
If \(\Qsymbol=0\), then \(u\) is supported on the first frequency block.  Since
\(n\ge5\) is odd, no three elements of \(\{1,-1\}\) sum to zero modulo
\(n\).  Hence \(\Avg{u^3}=0\).  Assume \(\Qsymbol>0\).

The zero-frequency cubic resonances are precisely
\[
  i+j=k,
  \qquad
  i+j+k=n,
  \qquad 1\le i,j,k\le m,
\]
together with conjugate sign patterns.  For \(2\le k\le m\), put
\[
  L_k=\sum_{\substack{i+j=k\\1\le i,j\le m}}a_i a_j,
  \qquad L_1=0,
\]
and for \(1\le k\le m\), put
\[
  R_k=\sum_{\substack{i+j=n-k\\1\le i,j\le m}}a_i a_j.
\]
Coefficient counting gives
\begin{equation}\label{eq:odd-resonance-bound}
  \abs{\Avg{u^3}}
  \le \frac1{\sqrt2}\sum_{k=1}^{m}a_k(3L_k+R_k).
\end{equation}
Indeed, a product corresponding to \(i+j=k\) has magnitude
\(a_i a_j a_k/(2\sqrt2)\); the three placements of the negative
frequency and the conjugate sign pattern give the contribution
\(3a_kL_k/\sqrt2\).  For \(i+j+k=n\), the ordered all-positive triples
and the all-negative conjugates give \(a_kR_k/\sqrt2\).

The boundary term \(k=1\) is not covered by the \(\Qsq\)-weight.  Since
\(R_1\) corresponds to \(i+j=2m\), the only ordered pair is \((m,m)\), and
\(R_1=a_m^2\).  Thus
\[
  B_0=\frac1{\sqrt2}a_1R_1
  =\frac1{\sqrt2}a_1a_m^2.
\]
Using
\[
  \Ssq\ge a_1^2+a_m^2\ge2a_1a_m,
  \qquad
  \Qsq\ge(m-1)a_m^2,
\]
we obtain
\begin{equation}\label{eq:odd-boundary}
  B_0\le \frac1{2\sqrt{2(m-1)}}\Ssq\Qsymbol
  \le \frac1{2\sqrt2}\Ssq\Qsymbol.
\end{equation}

It remains to estimate
\[
  B_1=\frac1{\sqrt2}\sum_{k=2}^{m}a_k(3L_k+R_k).
\]
By Cauchy--Schwarz with weight \(k-1\),
\[
  B_1\le \frac1{\sqrt2}\Qsymbol\left(
  \sum_{k=2}^{m}\frac{(3L_k+R_k)^2}{k-1}
  \right)^{1/2}.
\]
We claim that
\begin{equation}\label{eq:odd-W}
  \sum_{k=2}^{m}\frac{(3L_k+R_k)^2}{k-1}
  \le 11\Sfour.
\end{equation}
For \(2\le k\le m\), set
\[
  U_k=\sum_{\substack{i+j=k\\1\le i,j\le m}}a_i^2 a_j^2,
  \qquad
  V_k=\sum_{\substack{i+j=n-k\\1\le i,j\le m}}a_i^2 a_j^2.
\]
The sum \(L_k\) has \(k-1\) ordered terms, while \(R_k\) has \(k\)
ordered terms.  Hence
\[
  L_k^2\le(k-1)U_k,
  \qquad
  R_k^2\le kV_k.
\]
Since \(k/(k-1)\le2\),
\[
  \frac{(3L_k+R_k)^2}{k-1}
  \le \left(3\sqrt{U_k}+\sqrt2\sqrt{V_k}\right)^2
  \le 11(U_k+V_k).
\]
The \(U_k\) terms cover ordered pairs with \(2\le i+j\le m\), and the
\(V_k\) terms cover ordered pairs with \(m+1\le i+j\le2m-1\).  These
ranges are disjoint and contained in all ordered pairs \(1\le i,j\le m\).
Thus
\[
  \sum_{k=2}^{m}(U_k+V_k)
  \le \sum_{i,j=1}^{m}a_i^2 a_j^2=\Sfour,
\]
which proves \eqref{eq:odd-W}.  Therefore
\begin{equation}\label{eq:odd-B1}
  B_1\le \sqrt{\frac{11}{2}}\Ssq\Qsymbol.
\end{equation}
Combining \eqref{eq:odd-resonance-bound}, \eqref{eq:odd-boundary}, and
\eqref{eq:odd-B1},
\[
  \abs{\Avg{u^3}}
  \le \left(\sqrt{\frac{11}{2}}+\frac1{2\sqrt2}\right)\Ssq\Qsymbol
  <3\Ssq\Qsymbol.
\]
\end{proof}

\section{The third-moment estimate for even cyclic groups}

Let \(n=2m\ge4\).  The proof is parallel to the odd case, with one
change: the middle frequency \(m=n/2\) is self-conjugate.  Write a real
mean-zero function as
\[
  u=\sum_{j=1}^{m-1}\left(\coef_j\ch{j}+\overline{\coef_j}\ch{-j}\right)
  +b\ch{m},
  \qquad b\in\RR,
\]
where \(\ch{m}(x)=(-1)^x\). As before, real-valuedness pairs the nonmiddle
frequencies into conjugate pairs. The middle frequency is its own negative modulo \(2m\),
so its coefficient \(b\) is real.  Put
\[
  a_j=\sqrt2\,\abs{\coef_j}\quad (1\le j\le m-1),
  \qquad
  a_m=\abs{b}.
\]
Then
\[
  \Ssq=\sum_{j=1}^{m}a_j^2,
  \qquad
  \Qsq=\sum_{j=2}^{m}(j-1)a_j^2,
  \qquad \Ssymbol,\Qsymbol\ge0.
\]

\begin{proposition}[Even third-moment estimate]\label{thm:even-block}
For every real mean-zero \(u:\Zn{2m}\to\RR\), \(m\ge2\),
\[
  \abs{\Avg{u^3}}
  \le 3\Ssq\Qsymbol.
\]
\end{proposition}

\begin{proof}
The case \(\Qsymbol=0\) is again trivial, because then only the first frequency
block is present and no cubic zero-frequency resonance occurs for \(n\ge4\).
Assume \(\Qsymbol>0\).

For \(2\le k\le m-1\), define
\[
  L_k=\sum_{\substack{i+j=k\\1\le i,j\le m-1}}a_i a_j,
  \qquad
  R_k=\sum_{\substack{i+j=2m-k\\1\le i,j\le m-1}}a_i a_j,
\]
and define the middle-frequency sum
\[
  L_m=\sum_{\substack{i+j=m\\1\le i,j\le m-1}}a_i a_j.
\]
Empty sums are understood as zero.  The ordinary and wrap-around resonances give the
same coefficients as in the odd case, while the self-conjugate middle
frequency contributes with coefficient \(3\).  Hence
\begin{equation}\label{eq:even-resonance-bound}
  \abs{\Avg{u^3}}
  \le
  \frac1{\sqrt2}\sum_{k=2}^{m-1}a_k(3L_k+R_k)
  +3a_mL_m.
\end{equation}
By Cauchy--Schwarz, one has
\[
  \abs{\Avg{u^3}}
  \le \Qsymbol\sqrt{W},
\]
where
\[
  W=
  \sum_{k=2}^{m-1}\frac{(3L_k+R_k)^2}{2(k-1)}
  +\frac{9L_m^2}{m-1}.
\]
We prove
\begin{equation}\label{eq:even-W}
  W\le9\Sfour.
\end{equation}
For \(2\le k\le m-1\), let
\[
  U_k=\sum_{\substack{i+j=k\\1\le i,j\le m-1}}a_i^2 a_j^2,
  \qquad
  V_k=\sum_{\substack{i+j=2m-k\\1\le i,j\le m-1}}a_i^2 a_j^2.
\]
Both \(L_k\) and \(R_k\) have \(k-1\) ordered terms, so
\[
  L_k^2\le(k-1)U_k,
  \qquad
  R_k^2\le(k-1)V_k.
\]
Therefore
\[
  \frac{(3L_k+R_k)^2}{2(k-1)}
  \le \frac12\left(3\sqrt{U_k}+\sqrt{V_k}\right)^2
  \le5(U_k+V_k),
\]
using \((3a+b)^2\le10(a^2+b^2)\).  For the middle-frequency term, put
\[
  U_m=\sum_{\substack{i+j=m\\1\le i,j\le m-1}}a_i^2 a_j^2.
\]
Since \(L_m\) has \(m-1\) ordered terms,
\[
  \frac{9L_m^2}{m-1}\le9U_m.
\]
The \(U_k\) terms cover sums \(2\le i+j\le m-1\), the \(U_m\) term covers
sum \(m\), and the \(V_k\) terms cover sums \(m+1\le i+j\le2m-2\).  These
families partition the ordered pairs \(1\le i,j\le m-1\).  Hence
\[
  \sum_{k=2}^{m-1}(U_k+V_k)+U_m
  =\left(\sum_{j=1}^{m-1}a_j^2\right)^2.
\]
Consequently
\[
  W\le
  5\left(\sum_{j=1}^{m-1}a_j^2\right)^2+4U_m
  \le9\left(\sum_{j=1}^{m-1}a_j^2\right)^2
  \le9\Sfour.
\]
This proves \eqref{eq:even-W}, and therefore
\[
  \abs{\Avg{u^3}}
  \le \Qsymbol\sqrt{9\Sfour}=3\Ssq\Qsymbol.
\]
\end{proof}

\section{Proofs of the log-Sobolev inequalities}\label{sec:proofs}

\begin{proof}[Proof of Theorem \ref{thm:main}]
If \(n\ge4\) is odd, then \(n=2m+1\) with \(m\ge2\), and Proposition
\ref{thm:odd-block} verifies the hypothesis of Proposition
\ref{prop:block-to-cubic}.  If \(n\ge4\) is even, then Proposition
\ref{thm:even-block} verifies the same hypothesis.  Thus the cubic
criterion holds for all \(n\ge4\), and Proposition \ref{prop:cubic-criterion}
gives
\[
  \Ent{f^2}\le2\Avg{f\Apsi f}.
\]
The sharpness of the constant is a standard perturbation argument, since the spectral gap is 1. 
\end{proof}

We next give a proof on the circle for completeness.

\begin{proof}[Proof of Theorem \ref{thm:circle}]

We give a direct third-moment estimate. As in
Proposition \ref{prop:block-to-cubic}, it is enough to prove the circle analogue
of \eqref{eq:block-estimate-abstract}.  Let \(u\) be a real mean-zero
trigonometric polynomial and write
\[
  u=\sum_{j\ge1}\left(\coef_j\chi_j+\overline{\coef_j}\chi_{-j}\right),
  \qquad a_j=\sqrt2\,\abs{\coef_j}.
\]
Only finitely many of the \(a_j\)'s are nonzero.  Then
\[
  \Ssq=\AvgT{u^2}=\sum_{j\ge1}a_j^2,
  \qquad
  \Qsq=\AvgT{u\AT u}-\Ssq=\sum_{j\ge2}(j-1)a_j^2,
  \qquad \Ssymbol,\Qsymbol\ge0.
\]
If \(\Qsymbol=0\), only the first frequency block is present and no cubic
zero-frequency resonance occurs, so \(\AvgT{u^3}=0\).  Assume \(\Qsymbol>0\).
The only zero-frequency cubic resonances are \(i+j=k\), together with their
conjugates.  For \(k\ge2\), put
\[
  L_k=\sum_{i+j=k}a_i a_j,
  \qquad
  U_k=\sum_{i+j=k}a_i^2a_j^2,
  \qquad i,j\ge1.
\]
Coefficient counting gives
\[
  \abs{\AvgT{u^3}}
  \le \frac3{\sqrt2}\sum_{k\ge2}a_kL_k.
\]
By Cauchy--Schwarz, we obtain
\[
  \abs{\AvgT{u^3}}
  \le \frac3{\sqrt2}\Qsymbol
  \left(\sum_{k\ge2}\frac{L_k^2}{k-1}\right)^{1/2}.
\]
Since \(L_k\) has \(k-1\) ordered terms,
\[
  L_k^2\le (k-1)U_k.
\]
Therefore
\[
  \sum_{k\ge2}\frac{L_k^2}{k-1}
  \le \sum_{k\ge2}U_k
  =\sum_{i,j\ge1}a_i^2a_j^2=\Sfour.
\]
Thus
\[
  \abs{\AvgT{u^3}}
  \le \frac3{\sqrt2}\Ssq\Qsymbol
  \le 3\Ssq\Qsymbol.
\]
The cubic reduction gives
\(\EntT{f^2}\le2\mathcal E_{\TT}(f)\) for nonnegative trigonometric
polynomials \(f\).

For a general nonnegative \(f\in\mathcal D(\mathcal E_{\TT})\), let
\(F_N\) be the Fej\'er kernel and put \(f_N=F_N*f\). Then \(f_N\) is a
nonnegative trigonometric polynomial,
\[
  f_N\longrightarrow f\quad\text{in }L^2(\TT),
  \qquad
  \mathcal E_{\TT}(f_N)\longrightarrow\mathcal E_{\TT}(f).
\]
In particular, \(f_N^2\to f^2\) in \(L^1(\TT)\). Lower semicontinuity of
entropy therefore gives
\[
  \EntT{f^2}
  \le\liminf_{N\to\infty}\EntT{f_N^2}
  \le2\mathcal E_{\TT}(f). \qedhere
\]
\end{proof}

\section{Uniform quartic stability}\label{sec:stability}

In this section, we prove the stability of the log-Sobolev inequality. 
The main estimate is the following lemma, a quantitative remainder of the cubic majorant.

\begin{lemma}[Quartic remainder]\label{lem:quartic-remainder}
Define
\[
  R(t):=
  \frac23(t-1)^2(t+2)+(t^2-1)-2t^2\log t,
  \qquad t\ge0,
\]
where \(t^2\log t=0\) at \(t=0\).  Then
\begin{equation}\label{eq:quartic-pointwise}
  R(t)\ge
  \frac{(t-1)^4}{6\max\{1,t^2\}},
  \qquad t\ge0.
\end{equation}
Consequently, on any probability space, if \(f\ge0\) and
\(\norm{f}_2=1\), then
\begin{equation}\label{eq:quartic-averaged}
  \int R(f)\ge\frac1{12}\norm{f-1}_2^4.
\end{equation}
\end{lemma}

\begin{proof}
For \(t>0\),
\[
\begin{aligned}
  R'(t)&=2t^2-4t\log t-2,\\
  R''(t)&=4(t-1-\log t),\\
  R^{(3)}(t)&=4-\frac4t,\\
  R^{(4)}(t)&=\frac4{t^2}.
\end{aligned}
\]
Thus \(R(1)=R'(1)=R''(1)=R^{(3)}(1)=0\).
Taylor's theorem in Lagrange form therefore gives, for every \(t>0\),
a point \(\xi\) between \(1\) and \(t\) such that
\[
  R(t)
  =\frac{R^{(4)}(\xi)}{4!}(t-1)^4
  =\frac{(t-1)^4}{6\xi^2}.
\]
Here one may take \(\xi=1\) when \(t=1\).  Since
\(\xi^2\le\max\{1,t^2\}\), it follows that
\[
  R(t)
  \ge\frac{(t-1)^4}{6\max\{1,t^2\}}.
\]
At \(t=0\), it holds directly since \(R(0)=1/3\ge1/6\).
Moreover, Cauchy--Schwarz gives
\[
  \norm{f-1}_2^4
  \le
  \left(\int
    \frac{(f-1)^4}{\max\{1,f^2\}}\right)
  \left(\int\max\{1,f^2\}\right).
\]
Since
\[
  \int\max\{1,f^2\}\le 1+\int f^2=2,
\]
the pointwise estimate implies \eqref{eq:quartic-averaged}.
\end{proof}

\begin{proof}[Proof of Theorem~\ref{thm:stability-quartic} for \(d=1\)]
The already proved cubic criterion \eqref{eq:cubic-criterion} and the
definition of \(R\) give
\[
\begin{aligned}
  2\En(f)-\Ent{f^2}
  &=
  2\left(
    \En(f)-\frac13\Avg{(f-1)^2(f+2)}
  \right)
  +\Avg{R(f)}\\
  &\ge \Avg{R(f)}
  \ge \frac1{12}\norm{f-1}_2^4.
\end{aligned}
\]

It remains to prove the optimality of the exponent \(4\). This is the sharp
local exponent, although we do not claim that the coefficient
\(1/12\) is optimal.  Fix \(n\ge4\), and let \(v\) be a real
first eigenfunction of \(\Apsi\) normalized by
\[
  \Avg{v}=0,\qquad
  \Avg{v^2}=1,\qquad
  \Apsi v=v.
\]
Since \(v\) has Fourier support in \(\{\pm1\}\), one has
\(\Avg{v^3}=0\).  For sufficiently small \(\abs{\varepsilon}\),
the function
\[
  f_\varepsilon=\frac{1+\varepsilon v}{\sqrt{1+\varepsilon^2}}
\]
is nonnegative and satisfies \(\Avg{f_\varepsilon^2}=1\).  If
\(m_4=\Avg{v^4}\), then expansion at \(\varepsilon=0\) gives
\[
  2\En(f_\varepsilon)-\Ent{f_\varepsilon^2}
  =\frac{3+m_4}{6}\varepsilon^4
  +O(\abs{\varepsilon}^5),
\]
whereas
\[
  \norm{f_\varepsilon-1}_2^2
  =2\left(1-\frac1{\sqrt{1+\varepsilon^2}}\right)
  =\varepsilon^2+O(\varepsilon^4).
\]
Thus the deficit is of exact fourth order in the distance to the constants,
so no estimate of this type can hold with an exponent smaller than \(4\) and a
uniform positive coefficient.
\end{proof}

\begin{proof}[Proof of Theorem~\ref{thm:stability-quartic} for \(d>1\)]
Fix \(n\ge4\) and \(d>1\), and let \(G,\mu,A,f\) be as in the product
statement.  We
write points of \(G\) as \(x=(x_1,\ldots,x_d)\).  If
\(I=\{i_1,\ldots,i_r\}\), let \(\mathbb E_{x_{i_1},\ldots,x_{i_r}}\)
denote averaging in the displayed coordinates, with all other coordinates
held fixed, and put \(\mathbb E=\mathbb E_{x_1,\ldots,x_d}\).  An
expectation over an empty list is understood to be the identity.  Define
\[
  f_i(x_{i+1},\ldots,x_d)
  :=\left(\mathbb E_{x_1,\ldots,x_i}f^2\right)^{1/2},
  \qquad 0\le i\le d,
\]
where \(f_0=f\).  Then
\[
  f_i^2=\mathbb E_{x_i}f_{i-1}^2,
  \qquad f_d=1.
\]

Since \(\mathbb E f^2=1\) and \(f_d=1\), a telescoping sum gives
\begin{equation}\label{eq:product-entropy-telescope}
\begin{aligned}
  \operatorname{Ent}_{\mu}(f^2)
  &=\mathbb E\bigl(f^2\log f^2\bigr)\\
  &=\sum_{i=1}^d\left[
    \mathbb E_{x_i,\ldots,x_d}
      \bigl(f_{i-1}^2\log f_{i-1}^2\bigr)
    -\mathbb E_{x_{i+1},\ldots,x_d}
      \bigl(f_i^2\log f_i^2\bigr)\right]\\
  &=\sum_{i=1}^d\mathbb E_{x_{i+1},\ldots,x_d}
    \left[
      \mathbb E_{x_i}(f_{i-1}^2\log f_{i-1}^2)-(\mathbb E_{x_i}f_{i-1}^2)\log(\mathbb E_{x_i}f_{i-1}^2)\right].
\end{aligned}
\end{equation}
On a slice with \(f_i=0\), the integrand in the last line is defined to be
zero; the entire slice then vanishes.  In this and subsequent displays, the
quotient with denominator \(f_i^2\) is also understood to be zero on such a
slice.  Applying the unnormalized estimate
\eqref{eq:finite-quartic-stability-unnormalized} to every \(x_i\)-slice gives
\begin{equation}\label{eq:product-entropy-upper-bound}
  \operatorname{Ent}_{\mu}(f^2)
  \le 2\sum_{i=1}^d\mathbb E\bigl(f_{i-1}\Apsi^{(i)}f_{i-1}\bigr)
  -\frac1{12}\sum_{i=1}^d\mathbb E
       \frac{\bigl[\mathbb E_{x_i}(f_{i-1}-f_i)^2\bigr]^2}{f_i^2}.
\end{equation}
Indeed, the assertion about a zero slice follows from nonnegativity and
\(\mathbb E_{x_i}f_{i-1}^2=f_i^2\).

For fixed \((x_i,\ldots,x_d)\),
\[
  f_{i-1}(x_i,\ldots,x_d)
  =\norm{f(\mathord\cdot,x_i,\ldots,x_d)}
     _{L^2(x_1,\ldots,x_{i-1})}.
\]
Thus, for fixed \((x_{i+1},\ldots,x_d)\), the reverse triangle inequality
gives
\begin{align*}
  &\abs{f_{i-1}(x_i,x_{i+1},\ldots,x_d)
          -f_{i-1}(y_i,x_{i+1},\ldots,x_d)}^2\\
  &\quad\le
  \mathbb E_{x_1,\ldots,x_{i-1}}
  \abs{f(x_1,\ldots,x_i,\ldots,x_d)
          -f(x_1,\ldots,y_i,\ldots,x_d)}^2.
\end{align*}
Averaging against the jump kernel of \(\Apsi\) in the coordinate \(x_i\)
proves
\[
  \mathbb E\bigl(f\Apsi^{(i)}f\bigr)
  \ge
  \mathbb E_{x_i,\ldots,x_d}
    \bigl(f_{i-1}\Apsi^{(i)}f_{i-1}\bigr).
\]
After summing in \(i\), \eqref{eq:product-entropy-upper-bound} therefore
implies
\begin{equation}\label{eq:product-stronger-remainder}
  2\mathbb E(fAf)-\operatorname{Ent}_{\mu}(f^2)
  \ge \frac1{12}\sum_{i=1}^d
       \mathbb E
       \frac{\bigl[\mathbb E_{x_i}(f_{i-1}-f_i)^2\bigr]^2}{f_i^2}.
\end{equation}

Finally,
\[
  \mathbb E_{x_i}(f_{i-1}-f_i)^2
  =2f_i\bigl(f_i-\mathbb E_{x_i}f_{i-1}\bigr).
\]
This identity also holds when \(f_i=0\), in which case both sides vanish.
Jensen's inequality in the remaining variables, Cauchy--Schwarz in \(i\),
and one telescoping sum give
\begin{equation}\label{eq:product-remainder-telescope}
\begin{aligned}
  \sum_{i=1}^d\mathbb E
    \frac{\bigl[\mathbb E_{x_i}(f_{i-1}-f_i)^2\bigr]^2}{f_i^2}
  &=4\sum_{i=1}^d\mathbb E
    \bigl(f_i-\mathbb E_{x_i}f_{i-1}\bigr)^2\\
  &\ge\frac4d
    \left[\sum_{i=1}^d\mathbb E
      \bigl(f_i-\mathbb E_{x_i}f_{i-1}\bigr)\right]^2\\
  &=\frac4d\bigl(f_d-\mathbb E f\bigr)^2
   =\frac1d\norm{f-1}_{L^2(\mu)}^4.
\end{aligned}
\end{equation}
The last equality uses \(\mathbb E f^2=1\), which gives
\(\norm{f-1}_{L^2(\mu)}^2=2(1-\mathbb E f)\).  Combining
\eqref{eq:product-stronger-remainder} and
\eqref{eq:product-remainder-telescope} proves
\eqref{eq:product-quartic-stability}.

It remains to prove the optimality of the dimensional dependence. The
\(d^{-1}\) dependence in \eqref{eq:product-quartic-stability} is necessary, although
the numerical coefficient \(1/12\) is not asserted to be optimal.
Fix \(n\ge4\) and a first eigenfunction \(v\) as in the \(d=1\) part, and
set
\[
  g_\varepsilon
  =\frac{1+\varepsilon v}{\sqrt{1+\varepsilon^2}},
  \qquad
  F_\varepsilon(z_1,\ldots,z_d)
  =\prod_{i=1}^d g_\varepsilon(z_i).
\]
For sufficiently small \(\abs{\varepsilon}\), the function \(F_\varepsilon\)
is nonnegative and normalized in \(L^2(\pi^{\otimes d})\).
Writing \(m_4=\Avg{v^4}\) and
using the product generator \(A_n^{(d)}\) from
Theorem~\ref{thm:stability-quartic}, tensorization gives
\[
\begin{aligned}
  2\int_{\Zn{n}^d}F_\varepsilon A_n^{(d)}F_\varepsilon
    \,d\pi^{\otimes d}
    -\operatorname{Ent}_{\pi^{\otimes d}}(F_\varepsilon^2)
  &=d\left(2\En(g_\varepsilon)-\Ent{g_\varepsilon^2}\right)\\
  &=d\,\frac{3+m_4}{6}\varepsilon^4
    +O_d(\abs{\varepsilon}^5).
\end{aligned}
\]
On the other hand,
\[
  \norm{F_\varepsilon-1}_2^2
  =2\left(1-(1+\varepsilon^2)^{-d/2}\right)
  =d\varepsilon^2+O_d(\varepsilon^4).
\]
The ratio of the deficit to \(\norm{F_\varepsilon-1}_2^4\) therefore tends to
\((3+m_4)/(6d)\), so no dimension-independent coefficient can hold.
Together with \eqref{eq:product-quartic-stability}, this proves that the
\(d^{-1}\) dependence is optimal.
\end{proof}

\begin{remark}\label{rem:circle-stability}
Similar arguments give a quartic stability estimate for the circle group.
Also, it is possible to improve the numerical constant \(1/12\) using the
relaxed cubic-majorant method employed in~\cite{XieZhangFree2026}, but it
requires more efforts and we omit the discussion here.
\end{remark}


\begin{thebibliography}{JPPP17}

\bibitem[And02]{Andersson2002}
M. E. Andersson,
\emph{Beitrag zur Theorie des Poissonschen Integrals \"uber endlichen Gruppen},
Monatsh. Math. \textbf{134} (2002), no. 3, 177--190.

\bibitem[Bec75]{Beckner1975}
W. Beckner,
\emph{Inequalities in Fourier analysis},
Ann. of Math. (2) \textbf{102} (1975), no. 1, 159--182.


\bibitem[BJJ83]{BJJ1983}
W. Beckner, S. Janson and D. Jerison,
\emph{Convolution inequalities on the circle},
in \emph{Conference on harmonic analysis in honor of Antoni Zygmund}, Vol. I,
II, Wadsworth Math. Ser., Wadsworth, Belmont, CA, 1983, 32--43.


\bibitem[Bon70]{Bonami1970}
A. Bonami,
\emph{\'{E}tude des coefficients de Fourier des fonctions de \(L^p(G)\)},
Ann. Inst. Fourier (Grenoble) \textbf{20} (1970), no. 2, 335--402.

\bibitem[CFHQW26]{CaoFanHanQiuWang2026}
J. Cao, S. Fan, Y. Han, Y. Qiu and Z. Wang,
\emph{The optimal hypercontractive constants for \(\ZZ_3\) and biased
Bernoulli random variables},
arXiv:2602.17248v2, 2026.

\bibitem[CS03]{ChenSheu2003}
G.-Y. Chen and Y.-C. Sheu,
\emph{On the log-Sobolev constant for the simple random walk on the \(n\)-cycle:
the even cases},
J. Funct. Anal. \textbf{202} (2003), no. 2, 473--485.

\bibitem[CLS08]{ChenLiuSaloffCoste2008}
G.-Y. Chen, W.-W. Liu and L. Saloff-Coste,
\emph{The logarithmic Sobolev constant of some finite Markov chains},
Ann. Fac. Sci. Toulouse Math. (6) \textbf{17} (2008), no. 2, 239--290.

\bibitem[DS96]{DiaconisSaloffCoste1996}
P. Diaconis and L. Saloff-Coste,
\emph{Logarithmic Sobolev inequalities for finite Markov chains},
Ann. Appl. Probab. \textbf{6} (1996), no. 3, 695--750.

\bibitem[FF24]{FaustFawzi2024}
O. Faust and H. Fawzi,
\emph{Sum-of-squares proofs of logarithmic Sobolev inequalities on finite Markov chains},
IEEE Trans. Inform. Theory \textbf{70} (2024), no. 2, 803--819.


\bibitem[FI26]{FrankIvanisvili2026}
R. L. Frank and P. Ivanisvili,
\emph{The sharp log Sobolev inequality on finite cycles},
Adv. Math. \textbf{502} (2026), Part B, 111159;
arXiv:2605.29035, doi:10.1016/j.aim.2026.111159.

\bibitem[Gro75]{Gross1975}
L. Gross,
\emph{Logarithmic Sobolev inequalities},
Amer. J. Math. \textbf{97} (1975), no. 4, 1061--1083.

\bibitem[Gro06]{Gross2006}
L. Gross,
\emph{Hypercontractivity, logarithmic Sobolev inequalities, and applications: a
survey of surveys},
in \emph{Diffusion, quantum theory, and radically elementary mathematics},
Math. Notes \textbf{47}, Princeton Univ. Press, Princeton, NJ, 2006, 45--73.



\bibitem[JPPP17]{JPPP2017}
M. Junge, C. Palazuelos, J. Parcet and M. Perrin,
\emph{Hypercontractivity in group von Neumann algebras},
Mem. Amer. Math. Soc. \textbf{249} (2017), no. 1183, xii+83.



\bibitem[Rot80]{Rothaus}
O. S. Rothaus, 
\emph{Logarithmic Sobolev inequalities and the spectrum of Sturm--Liouville operators}, J. Funct.
Anal. \textbf{39} (1980), no. 1, 42--56.



\bibitem[Wei80]{Weissler1980}
F. B. Weissler,
\emph{Logarithmic Sobolev inequalities and hypercontractive estimates on the circle},
J. Funct. Anal. \textbf{37} (1980), no. 2, 218--234.

\bibitem[Wol07]{Wolff2007}
P. Wolff,
\emph{Hypercontractivity of simple random variables},
Studia Math. \textbf{180} (2007), no. 3, 219--236.

\bibitem[XZ26]{XZlean}
X. Xie and H. Zhang,
\emph{Lean verification of the main theorem in ``Sharp log-Sobolev inequalities
and quartic stability on finite cyclic groups''},
GitHub repository, commit \href{https://github.com/XinyuanXie-hub/cyclic-lsi-new/tree/43aee06f11fffe1d543ef57e564b7ff64c14e8f8}{43aee06}, 2026.

\bibitem[XZ26a]{XieZhangFree2026}
X. Xie and H. Zhang,
\emph{Sharp hypercontractivity for free group von Neumann algebras},
arXiv:2606.03423v3, 2026.

\bibitem[Yao25]{Yao2025}
G. Yao,
\emph{Optimal hypercontractivity and log-Sobolev inequalities on cyclic groups
\(\ZZ_{m\cdot2^k}\)},
arXiv:2512.03489, 2025.

\bibitem[Yao26]{Yao2026}
G. Yao,
\emph{A complete solution of optimal hypercontractivity and Log--Sobolev
inequalities on cyclic groups \(\ZZ_n\) for \(n\ge4\)},
arXiv:2605.30867v1, 2026.

\end{thebibliography}
\end{document}